\documentclass[compmod]{rsl}
\gridframe{N}
\usepackage{amsmath,amssymb,latexsym,natbib,amsfonts,mathrsfs}
\volume{0}
\issue{0}

\pyear{2012}
\pmonth{Month}
\doinu{10.1017/S1755020300000000}

\begin{document}
\leftrunninghead{Samuel A. Alexander}
\rightrunninghead{Global Necessitation}

\title[Global Necessitation]{The Rule of Global Necessitation}

\author[S. A. Alexander]{SAMUEL A. ALEXANDER}
\affil{The Ohio State University\footnote{Received October 2012}}

\maketitle

\begin{abstract}
For half a century, authors have
weakened the rule of necessitation in various more or less ad hoc ways
in order to make inconsistent systems consistent.
More recently, necessitation was weakened in a systematic
way, not for the purpose of resolving paradoxes
but rather to salvage the deduction theorem for modal logic.
We show how this systematic weakening can be applied
to the older problem of paradox resolution.
Four examples are given:
a predicate symbol $S4$ consistent with arithmetic; a resolution of
the surprise examination paradox; a resolution of Fitch's paradox;
and finally, the construction of a knowing machine which knows its own code.
We discuss a technique for possibly finding answers to a question of
P.~\'{E}gr\'{e} and J.~van Benthem.
\end{abstract}

\section{Introduction}

To make a certain system consistent, \cite{myhill} suggested (pp.~469-470) weakening the rule of necessitation to only range over arithmetical
formulas.
Since then, others have suggested various other weakenings of necessitation to make other systems consistent.
\cite{halbach2008}
restricted necessitation to formulas not involving a particular predicate; \cite{egre} hinted (p.~44) at separating necessitation from
soundness.
Meanwhile, \cite{fitting} weakened necessitation in a different way, not to repair inconsistencies but instead to recover the deduction theorem
for modal logic (see \cite{hakli} for a survey on this issue).
Fitting's weakening has the advantage that it is very systematic.  We apply it to the older objective of fixing inconsistencies.

Suppose we have some axioms which are divided into \emph{global} and \emph{local} axioms.
The \emph{rule of global necessitation} for a modal operator $K$ is as follows:
\[
\frac{\phi}{K(\phi)}\mbox{ provided $\phi$ has been proved without appeal to global axioms.}
\]
The rule of global necessitation for a predicate symbol $K$ is the same, except that
$\phi/K(\phi)$ is replaced by $\phi/K(\ulcorner\phi\urcorner)$.
The idea is that a global axiom is one which is thoroughly trusted by the agent whose knowledge (or belief or...)
is represented by $K$, whereas a local axiom is true, but may not be known by the agent, or may be known with low
conviction.

Formally, if a system $S$ consists of a set of global axioms, a set of local axioms,
and the rule of global necessitation for $K$ (and modus ponens, which we hereafter implicitly include in every system),
we say $S\models\phi$ if 
there is a sequence $\phi_1,\ldots,\phi_n$ such that $\phi_n$ is $\phi$ and
such that for every $i$,
\begin{enumerate}
\item $\phi_i$ is a local axiom, or
\item $\phi_i$ is a global axiom, or
\item $\phi_i$ is logically valid, or
\item $\phi_i$ follows from two
earlier members of the sequence by modus ponens, or
\item $\phi_i$ is $K(\phi_j)$ (or $K(\ulcorner\phi_j\urcorner)$ if $K$ is a predicate symbol)
for some $j<i$ such that for every $1\leq k\leq j$,
$\phi_k$ is an instance of one of items (2)-(5) of this list.
\end{enumerate}

\clearpage
For example, the system $GLS$ of provability logic (\cite{boolos} p.~65) is equivalent to a system with global axioms the axioms of $GL$,
local axiom schema $T$, and the rule of global necessitation.

The idea for this weakened necessitation appeared in \cite{fitting} (p.~94)
and
more recently it was used by \cite{hakli}.  In both cases it was simply called the rule of necessitation;
we have called it the rule of global necessitation so as to distinguish it from its stronger ancestor.

The following lemma (compare the suggestion of \cite{smorynski} (p.~454),
as described by \cite{hakli} (p.~854)) is very useful for showing consistency of systems involving the rule of global necessitation.
To state it succinctly, we adopt the following notation: if $K$ is a modal operator, $K[\phi]$ will denote
$K(\phi)$, and if $K$ is a predicate symbol, $K[\phi]$ will denote $K(\ulcorner\phi\urcorner)$.

\begin{lem}
\label{standardizer}
Assume a logic where the compactness theorem holds.
Suppose $S$ is a system consisting of a set $S_g$ of global axioms, a set $S_{\ell}$ of local axioms,
and the rule of global necessitation for $K$.
Let $S'$ be the following system of axioms:
\begin{enumerate}
\item $\phi$, if $\phi\in S_g$.
\item $K[\phi]$ if $\phi$ is logically valid.
\item $K[\phi\rightarrow\psi]\rightarrow K[\phi]\rightarrow K[\psi]$.
\item $K[\phi]$ if $\phi$ is an instance of (1)-(3) or (recursively) (4).
\item $\phi$, if $\phi\in S_{\ell}$.
\end{enumerate}
For any $\phi$, if $S\models\phi$, then $S'\models\phi$.
\end{lem}

\begin{proof}
Let $S_0$ be the system consisting of axioms $S_g$ and the rule of global necessitation.  Let $S'_0$ be the set of axioms
from lines (1)--(4) of $S'$.  We claim:  For any $\phi$, if $S_0\models\phi$ then $S'_0\models\phi$.
This is proved by induction on proof length from $S_0$.
The only interesting case is when $\phi$ is $K[\phi_0]$, where $S_0\models\phi_0$ in fewer steps.
By induction, $S'_0\models\phi_0$.  By compactness, there are $s_1,\ldots,s_n\in S'_0$
such that
\[
s_1\rightarrow \cdots \rightarrow s_n \rightarrow\phi_0
\]
is valid.  By (2),
\[
S'_0\models K[ s_1\rightarrow \cdots \rightarrow s_n \rightarrow\phi_0 ].\]
By repeated applications of (3),
\[
S'_0\models K[s_1] \rightarrow \cdots \rightarrow K[s_n] \rightarrow K[\phi_0].\]
By (4), for each $i$, $K[s_i]\in S'_0$, since each $s_i\in S'_0$ and $S'_0$ is closed under $K$.
Thus $S'_0\models K[\phi_0]$, as desired.

Now we attack the main lemma itself, again by induction on proof length.
Suppose $S\models\phi$.
There is only one nontrivial case: $\phi$ is $K[\phi_0]$
where $S\models\phi_0$ in fewer steps and $\phi$ is obtained from that shorter proof
by the rule of global necessitation.
Since $\phi$ is so obtained, this means every step in the proof of $\phi_0$ is
an instance of lines (2)-(5) of the definition of proof on the previous page.
This implies $S_0\models\phi_0$, and thus $S_0\models K[\phi_0]$,
and so by the claim, $S'_0\models K[\phi_0]$, so certainly $S'\models K[\phi_0]$.
\end{proof}

The bulk of the paper will concern modal operator paradoxes.  The reason for this is that,
due to self-reference, systems involving full necessitation for a predicate symbol blow up
extremely easily\footnote{In \cite{montague} (p.~294)
(also quoted in \cite{egre}) we find: ``...if necessity is to be treated 
syntactically ... then virtually all of modal logic, even the weak 
system $S1$, is to be sacrificed.''}, and thus the well-known 
paradoxes (see \cite{egre}) tend to be low-level, by which I mean they usually derive contradiction
from some
set of assumptions which is (equivalent to) a tiny fragment of a predicate-symbol $S4$\footnote{One exception
is the predicate symbol treatment of the surprise examination paradox, but that paradox has an equally good
modal operator treatment anyway.  Another exception appears in \cite{horsten}, more on that in our Conclusion below.}.
Thus, the following theorem more or less resolves them all in one fell swoop (and might, therefore,
be a tentative step toward resolving tensions described in sections 1 \& 2 of \cite{halbach}).

\begin{thm}
\label{predicates4}
(A predicate symbol version of weakened $S4$, consistent with arithmetic)
The following system is consistent (in the language of Peano Arithmetic extended by a predicate
symbol $K$):
\begin{enumerate}
\item (Global) The axioms of Peano Arithmetic.
\item (Global) $K(\ulcorner\phi\rightarrow\psi\urcorner)\rightarrow K(\ulcorner\phi\urcorner)\rightarrow K(\ulcorner\psi\urcorner)$.
\item (Global) $K(\ulcorner\phi\urcorner)\rightarrow K(\ulcorner K(\ulcorner\phi\urcorner) \urcorner)$.
\item (Local) $K(\ulcorner\phi\urcorner)\rightarrow\phi$.
\item The rule of global necessitation.
\end{enumerate}
\end{thm}

(We'll discuss the philosophical plausibility of localizing soundness like this in the Conclusion.)

\begin{proof}
Let $S'$ be the following system:
\begin{enumerate}
\item The axioms of Peano Arithmetic.
\item $K(\ulcorner\phi\rightarrow\psi\urcorner)\rightarrow K(\ulcorner\phi\urcorner)\rightarrow K(\ulcorner\psi\urcorner)$.
\item $K(\ulcorner\phi\urcorner)\rightarrow K(\ulcorner K(\ulcorner\phi\urcorner) \urcorner)$.
\item $K(\ulcorner\phi\urcorner)$ whenever $\phi$ is valid.
\item $K(\ulcorner\phi\urcorner)$ whenever $\phi$ is an instance of (1)--(4) or (recursively) (5).
\item $K(\ulcorner\phi\urcorner)\rightarrow\phi$.
\end{enumerate}
By Lemma~\ref{standardizer}, we need only show $S'$ is consistent.

In the absense of non-modus ponens rules of inference, consistency is easy to prove: merely construct a model.
Let $S'_0$ consist of the axioms in lines (1)-(5) of $S'$.
Let $\mathscr{N}$ be the model which has universe $\mathbb{N}$, which interprets symbols of PA in the intended ways,
and which interprets $K$ as follows:
\[
\mbox{$\mathscr{N}\models K(\ulcorner\phi\urcorner)$ iff $S'_0\models\phi$}
\]
(we do not care how $\mathscr{N}$ interprets $K(\overline{n})$ if $n$ is not the G\"{o}del number of a formula).
We will show that $\mathscr{N}\models S'$, proving the theorem.

Preliminary Claim: whenever $S'_0\models\phi$, $S'_0\models K(\ulcorner\phi\urcorner)$.
To see this, suppose $S'_0\models\phi$.  By compactness, there are $s_1,\ldots,s_n\in S'_0$
such that
\[s_1\rightarrow\cdots\rightarrow s_n\rightarrow\phi\]
is valid.  It follows $S'_0\models K(\ulcorner\phi\urcorner)$ by an argument similar to
the proof of Lemma~\ref{standardizer}.

Armed with the claim, we show $\mathscr{N}\models S'$.  Suppose $\sigma\in S'$, we will show $\mathscr{N}\models \sigma$.

Case 1: $\sigma$ is an axiom of Peano Arithmetic.  Then $\mathscr{N}\models\sigma$ because $\mathscr{N}$ has universe $\mathbb{N}$
and interprets symbols of PA in their intended ways.

Case 2: $\sigma$ is $K(\ulcorner\phi\rightarrow\psi\urcorner)\rightarrow K(\ulcorner\phi\urcorner)\rightarrow K(\ulcorner\psi\urcorner)$.
Assume $\mathscr{N}\models K(\ulcorner\phi\rightarrow\psi\urcorner)$
and $\mathscr{N}\models K(\ulcorner\phi\urcorner)$.  By definition this means $S'_0\models \phi\rightarrow\psi$
and $S'_0\models \phi$.  By modus ponens, $S'_0\models \psi$, so $\mathscr{N}\models  K(\ulcorner\psi\urcorner)$, as desired.

Case 3: $\sigma$ is $K(\ulcorner\phi\urcorner)\rightarrow K(\ulcorner K(\ulcorner\phi\urcorner) \urcorner)$.
Assume $\mathscr{N}\models K(\ulcorner\phi\urcorner)$.
This means $S'_0\models \phi$.  By the preliminary claim, $S'_0\models K(\ulcorner\phi\urcorner)$,
which shows $\mathscr{N}\models K(\ulcorner K(\ulcorner\phi\urcorner) \urcorner)$.

Case 4: $\sigma$ is $K(\ulcorner\phi\urcorner)$ where $\phi$ is valid.
Since $\phi$ is valid, certainly $S'_0\models \phi$, so $\mathscr{N}\models K(\ulcorner\phi\urcorner)$.

Case 5: $\sigma$ is $K(\ulcorner\phi\urcorner)$ where $\phi$ is an instance of (1)-(5) of $S'$.
Then $\phi\in S'_0$, so $S'_0\models \phi$, so $\mathscr{N}\models K(\ulcorner\phi\urcorner)$.

Case 6: $\sigma$ is $K(\ulcorner\phi\urcorner)\rightarrow\phi$.
Assume $\mathscr{N}\models K(\ulcorner\phi\urcorner)$, which means $S'_0\models\phi$.
Note that, by Cases 1--5, $\mathscr{N}\models S'_0$.
Since $\mathscr{N}\models S'_0$ and $S'_0\models\phi$,
$\mathscr{N}\models \phi$, completing the proof.
\end{proof}

Although we expressed Theorem~\ref{predicates4} as a result about knowledge,
it could just as well be couched as a result about truth predicates
(thickening the plot of \cite{friedman}).

\section{Myhill's Necessitation and a Moore's Paradox}

\cite{myhill} made a certain system consistent by forcing necessitation's premises
to be arithmetical.  This differs from our approach:
Myhill's weak necessitation discriminates based on the \emph{form} of the premise, whereas ours 
discriminates on the \emph{origin} of the premise.
In this section, we exhibit a paradox
thwarting Myhill's method but repairable with global necessitation.

Let $\Psi$ be a sentence (in the language of Peano arithmetic) which
is true about $\mathbb{N}$ but independent of PA.  For example $\Psi$ 
could be $Con(PA)$ or $\Psi$ could be Goodstein's Theorem (see 
\cite{kirby}).  
Thus the Peano arithmetist does not know $\Psi$, although $\Psi$ is true.

We might, therefore, attempt to study arithmetists' knowledge using a system
consisting of the axioms of PA, $\Psi$, $\neg K(\ulcorner\Psi\urcorner)$, and the rule of
necessitation.  Of course this would be very flawed (and provably inconsistent
with a very short proof): the rule of necessitation
should not be allowed to interact with $\Psi$ because the arithmetist does not know $\Psi$!
None of this is surprising, the only reason we bring it up is to point out that
Myhill's weakening of necessitation, restricting it to purely arithmetical premises,
has no effect here.  On the other hand, the rule of global necessitation can
be used to articulate the system properly.

\begin{thm}
\label{moore1}
The following system is consistent ($K$ a predicate symbol):
\begin{enumerate}
\item (Global) The axioms of Peano Arithmetic.
\item (Local) $\Psi$.
\item (Local) $\neg K(\ulcorner\Psi\urcorner)$.
\item The rule of global necessitation.
\end{enumerate}
\end{thm}

Though the theorem may rightly be considered silly, the proof introduces a useful trick.

\begin{proof}
By Lemma~\ref{standardizer}, it suffices to show the following system $S'$ is consistent:
\begin{enumerate}
\item The axioms of Peano Arithmetic.
\item $K(\ulcorner\phi\urcorner)$ whenever $\phi$ is valid.
\item $K(\ulcorner\phi\rightarrow\psi\urcorner)\rightarrow K(\ulcorner\phi\urcorner)\rightarrow K(\ulcorner\psi\urcorner)$.
\item $K(\ulcorner\phi\urcorner)$ whenever $\phi$ is an instance of (1)-(4).
\item $\Psi$.
\item $\neg K(\ulcorner\Psi\urcorner)$.
\end{enumerate}
Let $S'_0$ be the set of axioms in lines (1)-(4) of $S'$.  Let $\mathscr{N}$ be a model with universe $\mathbb{N}$,
interpreting symbols of PA in the intended ways, and interpreting knowledge so that
\[
\mbox{$\mathscr{N}\models K(\ulcorner\phi\urcorner)$ iff $S'_0\models\phi$.}
\]
We will show $\mathscr{N}\models S'$, proving the theorem.  The tricky part is to show $\mathscr{N}\models \neg K(\ulcorner\Psi\urcorner)$.
For that we need:

Preliminary Claim: $S'_0\not\models\Psi$.

In the absense of non-modus ponens rules of inference, to show a theory does not prove a formula, it suffices
to build a model of the theory where the formula is untrue.
Here, it suffices to build a model of $S'_0$ where $\Psi$ is untrue.

Since $\Psi$ is independent of PA, there is a (nonstandard) model $\mathscr{M}$, in the language of PA,
such that $\mathscr{M}\models PA$ and $\mathscr{M}\models \neg\Psi$.
We will extend $\mathscr{M}$ to a model $\mathscr{M}'$ in the language of PA plus $K$.
To do so we must specify how $\mathscr{M}'$ is to interpret $K$.  Let $\mathscr{M}'$ interpret $K$ so that
\[
\mbox{$\mathscr{M}'\models K(\ulcorner\phi\urcorner)$ for every $\phi$}
\]
(loosely: $\mathscr{M}'$ ``knows everything'').
I claim $\mathscr{M}'\models S'_0$.  To see this, let $\tau\in S'_0$, we will show $\mathscr{M}'\models\tau$.

Case 1: $\tau$ is an axiom of PA.  Then $\mathscr{M}'\models\tau$ since $\mathscr{M}\models\tau$ and $\mathscr{M}'$ agrees
with $\mathscr{M}$ on the language of PA.

Case 2: $\tau$ is $K(\ulcorner\phi\urcorner)$ where $\phi$ is valid.
Then $\mathscr{M}'\models K(\ulcorner\phi\urcorner)$ since $\mathscr{M}'$ knows everything.

Case 3: $\tau$ is $K(\ulcorner\phi\rightarrow\psi\urcorner)\rightarrow K(\ulcorner\phi\urcorner)\rightarrow K(\ulcorner\psi\urcorner)$.
We have $\mathscr{M}'\models K(\ulcorner\psi\urcorner)$ since $\mathscr{M}'$ knows  everything.

Case 4: $\tau$ is $K(\ulcorner\phi\urcorner)$ where $\phi$ is an instance of (1)-(4).
Then $\mathscr{M}'\models K(\ulcorner\phi\urcorner)$ since $\mathscr{M}'$ knows  everything.

Thus $\mathscr{M}'\models S'_0$.  And $\mathscr{M}'\not\models \Psi$ since $\mathscr{M}\not\models\Psi$ and
$\mathscr{M}'$ agrees with $\mathscr{M}$ on the language of PA.
This proves the preliminary claim (in fact by arbitrariness of $\Psi$ it proves
$S'_0$ is a conservative extension of PA).

For the theorem itself, let $\sigma\in S'$, I claim $\mathscr{N}\models\sigma$.

Case 1: $\sigma$ is an axiom of Peano Arithmetic.
Then $\mathscr{N}\models\sigma$ since $\mathscr{N}$ extends $\mathbb{N}$.

Case 2: $\sigma$ is $K(\ulcorner\phi\urcorner)$ where $\phi$ is valid.
Since $\phi$ is valid, $S'_0\models\phi$, thus $\mathscr{N}\models K(\ulcorner\phi\urcorner)$.

Case 3: $\sigma$ is $K(\ulcorner\phi\rightarrow\psi\urcorner)\rightarrow K(\ulcorner\phi\urcorner)\rightarrow K(\ulcorner\psi\urcorner)$.
Assume $\mathscr{N}\models K(\ulcorner\phi\rightarrow\psi\urcorner)$
and $\mathscr{N}\models K(\ulcorner\phi\urcorner)$.  That means $S'_0\models\phi\rightarrow\psi$ and $S'_0\models \phi$.
By modus ponens, $S'_0\models\psi$, so $\mathscr{N}\models K(\ulcorner\psi\urcorner)$.

Case 4: $\sigma$ is $K(\ulcorner\phi\urcorner)$ where $\phi$ is an instance of (1)-(4).  Then $\phi\in S'_0$, hence $S'_0\models\phi$
and $\mathscr{N}\models K(\ulcorner\phi\urcorner)$.

Case 5: $\sigma$ is $\Psi$.  We have $\mathscr{N}\models\Psi$ since $\mathscr{N}$ extends $\mathbb{N}$.

Case 6: $\sigma$ is $\neg K(\ulcorner\Psi\urcorner)$.  By the Preliminary Claim, $S'_0\not\models\Psi$.
Thus $\mathscr{N}\not\models K(\ulcorner\Psi\urcorner)$, as desired, proving the theorem.
\end{proof}

Whenever we establish a consistency result by localizing certain schemas, a natural question is what is the
minimum set of schemas we can localize and obtain consistency.  
It is clear that we cannot strengthen Theorem~\ref{moore1}
by globalizing $\Psi$.  And by L\"{o}b's Theorem it follows
we cannot globalize $\neg K(\ulcorner\Psi\urcorner)$ either,
so in some sense Theorem~\ref{moore1} is sharp.
Remarkably, though, the same system, formulated using a modal operator 
instead of a predicate symbol, can be further sharpened.

\begin{thm}
\label{moore2}
The following system is consistent ($K$ a modal operator):
\begin{enumerate}
\item (Global) The axioms of Peano Arithmetic.
\item (Local) $\Psi$.
\item (Global) $\neg K(\Psi)$.
\item The rule of global necessitation.
\end{enumerate}
\end{thm}

To prove the theorem, we need a technical lemma which we'll use again 
later.

\begin{lem}
\label{techlemma}
Suppose $K$ is a modal operator, $S$ is a set of axioms, and $K(S)$
(the $K$-closure of $S$) is the set of axioms
\begin{enumerate}
\item $\phi$, if $\phi\in S$.
\item $K(\phi)$, if $\phi$ is an instance of (1) or (recursively) (2).
\end{enumerate}
Suppose $\mathscr{N}$ is a model in which $K$ is interpreted so that
$\mathscr{N}\models K(\phi)$ iff $\mathscr{N}\models\phi$.
If $\mathscr{N}\models S$ then $\mathscr{N}\models K(S)$.
\end{lem}

\begin{proof}
Assume $\mathscr{N}\models S$.
We prove by induction on formula complexity that for all $\phi\in K(S)$,
$\mathscr{N}\models \phi$.

Case 1: $\phi$ is of the form $K(\phi_0)$ for some $\phi_0$.
If $\phi\in S$ we are done.  If
$\phi\not\in S$, then $\phi$ is an instance of line (2) of $K(S)$,
and thus $\phi_0\in K(S)$.  By induction, $\mathscr{N}\models \phi_0$,
and thus $\mathscr{N}\models K(\phi_0)$.

Case 2: $\phi$ is not of the form $K(\phi_0)$.  Then $\phi$
must be an instance of line (1) of $K(S)$, i.e., $\phi\in S$, and
$\mathscr{N}\models\phi$ by assumption.
\end{proof}

Like Theorem~\ref{moore1}, the proof of Theorem~\ref{moore2}
introduces a trick which we'll use again later.

\begin{proof}[Proof of Theorem~\ref{moore2}]
By Lemma~\ref{standardizer}, it suffices to show the following system $S'$ is consistent:
\begin{enumerate}
\item The axioms of Peano Arithmetic.
\item $\neg K(\Psi)$.
\item $K(\phi)$ whenever $\phi$ is valid.
\item $K(\phi\rightarrow\psi)\rightarrow K(\phi)\rightarrow K(\psi)$.
\item $K(\phi)$ whenever $\phi$ is an instance of (1)-(5).
\item $\Psi$.
\end{enumerate}
Let $S'_0$ consist of lines (1)-(5).  Let $\mathscr{N}$ be a model with universe $\mathbb{N}$, interpreting symbols of PA in
the intended ways, and interpreting $K$ so that
\[
\mbox{$\mathscr{N}\models K(\phi)$ iff $S'_0\models\phi$}.\]
We'll show $\mathscr{N}\models S'$, proving the theorem.
As before, the tricky part is showing $\mathscr{N}\models \neg K(\Psi)$.
For that we'll need the same preliminary claim as before, but now with a different proof.

Preliminary Claim:
$S'_0\not\models \Psi$.

To prove $S'_0\not\models\Psi$ it's enough to build a model of $S'_0$ where $\Psi$ fails.
Since $PA\not\models\Psi$, there is a (nonstandard) model $\mathscr{M}$ (in the language of PA) such that
$\mathscr{M}\models PA$ and $\mathscr{M}\not\models\Psi$.  Extend $\mathscr{M}$ to a model
$\mathscr{M}'$ in the language of PA plus $K$, by recursively letting
\[
\mbox{$\mathscr{M}'\models K(\phi)$ iff $\mathscr{M}'\models\phi$}
\]
for every formula $\phi$ (this would, of course, be impossible if $K$ were a predicate symbol).
I claim $\mathscr{M}'\models S'_0$.  To see this, let $\tau\in S'_0$, we'll show $\mathscr{M}'\models\tau$.

Case 1: $\tau$ is an axiom of PA.  Then $\mathscr{M}'\models \tau$ since $\mathscr{M}\models\tau$
and $\mathscr{M}'$ agrees with $\mathscr{M}$ on the language of PA.

Case 2: $\tau$ is $\neg K(\Psi)$.  By definition, $\mathscr{M}'\models K(\Psi)$ iff $\mathscr{M}'\models\Psi$.
And $\mathscr{M}'\not\models\Psi$ since $\mathscr{M}\not\models\Psi$ and they agree on the language of PA.
So $\mathscr{M}'\models\neg K(\Psi)$.

Case 3: $\tau$ is $K(\phi)$ where $\phi$ is valid.  Since $\phi$ is valid, $\mathscr{M}'\models\phi$,
thus $\mathscr{M}'\models K(\phi)$.

Case 4: $\tau$ is $K(\phi\rightarrow\psi)\rightarrow K(\phi)\rightarrow K(\psi)$.
Assume $\mathscr{M}'\models K(\phi\rightarrow\psi)$ and $\mathscr{M}'\models K(\phi)$.
This means $\mathscr{M}'\models \phi\rightarrow\psi$ and $\mathscr{M}'\models\phi$.
Thus $\mathscr{M}'\models\psi$, so $\mathscr{M}'\models K(\psi)$.

Case 5: $\tau$ is $K(\phi)$ where $\phi$ is an instance of (1)-(5).
Then $\mathscr{M}'\models\tau$ by
Lemma~\ref{techlemma}.

This establishes $\mathscr{M}'\models S'_0$.  And $\mathscr{M}'\not\models\Psi$, since $\mathscr{M}\not\models\Psi$
and the two agree on the language of PA, therefore $S'_0\not\models\Psi$ and the claim is proved.

Now back to the main theorem, for $\sigma\in S'$, we must show $\mathscr{N}\models\sigma$.
This splits the proof into the same six cases as Theorem~\ref{moore1}, and the six cases are proved
in exactly the same ways as in Theorem~\ref{moore1}, so we omit them.  The crucial difference
was the different proof of the Preliminary Claim.
\end{proof}

This has an interesting informal application to non-idealized epistemology.
As a mathematician, I would like to be able to state the following Moore's paradox:
\begin{itemize}
\item Peano Arithmetic (or ZFC) is true,
\item The Four-Color Theorem is true (I trust \cite{appel}),
\item I do not know that the Four-Color Theorem is true (I have utmost confidence I will not discover a proof of it in my lifetime),
\end{itemize}
and operate using the rule of global necessitation ($\neg K(4CT)$ being global, and $4CT$ being local).
Strictly speaking, this is inconsistent;
but 
we won't obtain contradiction through purely epistemological
means, as the system would be consistent if $4CT$ were independent.  I believe I can work in this system my whole life
and not state a contradiction.
See \cite{shapiroslides} p.~7 for a similar example.

The proofs of Theorems~\ref{moore1} \& \ref{moore2} highlight a difference between operators and predicates:
with operators, we can prove $\not\models$ by constructing toy models where the operator
is interpreted as truth, which is impossible with predicates.

\section{Example: The Surprise Examination Paradox}

The surprise examination paradox hardly needs an introduction (but the 
survey by \cite{chow} is worthy of pleasure-reading).
For a formalization,
we turn to
\cite{mclelland}, where we find the following system laid out on 
pp.~76--77 (the language is propositional, with atoms $p_1,\ldots,p_n$, 
where $n$ is the number of days in the week and $p_i$ is read as ``the 
examination takes place on the $i$th day'', further extended by a modal 
operator $K$, with $K(\phi)$ read as ``$\phi$ is known just prior to the 
exam''):
\begin{enumerate}
\item $K(\phi)\rightarrow\phi$.
\item $K(\phi\rightarrow\psi)\rightarrow K(\phi)\rightarrow K(\psi)$.
\item $T_n$, by which is meant $p_1\vee\cdots\vee p_n$ (``an exam will occur'').
\item $\bigwedge_{i=1}^n \neg K(p_i)$ (``it will be a surprise'').
\item $\bigwedge_{i=1}^n ((\neg T_i)\rightarrow K(\neg T_i))$,
where $T_i$ abbreviates $p_1\vee\cdots\vee p_i$ (``exam non-occurrences 
are observable'').
\item The rule of necessitation.
\end{enumerate}

The above system is inconsistent, due to the surprise 
examination paradox.  We would like to make it consistent by weakening 
necessitation.  Following the example in Section 1 in which we made $S4$
consistent for predicate symbols by localizing soundness, we might
attempt to resolve the surprise examination paradox in the same way.
But it is not difficult to see that the above system remains inconsistent
even with $K(\phi)\rightarrow\phi$ localized and necessitation replaced by global
necessitation (in fact, the system remains inconsistent even if the
$K(\phi)\rightarrow\phi$ schema is completely removed).

It is possible to show the system becomes consistent if we localize \emph{both}
$K(\phi)\rightarrow\phi$ and $\bigwedge_{i=1}^n \neg K(p_i)$.  This is unsatisfactory, though,
because to localize the latter axiom is to suggest the students ignored the
``surprise'' part of the surprise examination announcement or didn't take it very seriously.
Fortunately, \cite{kritchman}
come to our rescue.
When McLelland and Chihara laid down their formulation of the paradox,
they did so with the understanding that an event is \emph{surprising} 
precisely if it is unknown just prior to its occurrence.
Kritchman and Raz point out that, under this definition of surprise, an 
inconsistent knower is never surprised, since an 
inconsistent knower knows everything (the same was observed by 
\cite{halpern} some time earlier).  It seems that if I know 
\emph{everything} (including contradictions), that's as bad as knowing 
nothing whatsoever.  Therefore, we should 
admit a different definition of surprise:
\begin{quote}
An event is surprising if either 1) it is unknown just prior to its 
occurrence, or 2) a contradiction is known.
\end{quote}
We know (locally) that the students in the surprise exam paradox are 
sound, so to us, the two definitions of surprise are equivalent.
But the students do not necessarily know as much.
Based on this, the axiom $\bigwedge_{i=1}^n \neg K(p_i)$
should be weakened to $\left(\bigwedge_{i=1}^n \neg K(p_i)\right)\vee 
K(\bot)$ where $\bot$ is some contradiction (such as $p_1\wedge \neg 
p_1$).
Having done so, we can keep it global, obtaining a more satisfactory
resolution to the paradox.

\begin{thm}
\label{surpriseexamthm}
Assume $n>1$.  The following system is consistent:
\begin{enumerate}
\item (Local) $K(\phi)\rightarrow\phi$.
\item (Global) $K(\phi\rightarrow\psi)\rightarrow K(\phi)\rightarrow K(\psi)$.
\item (Global) $T_n$, by which is meant $p_1\vee\cdots\vee p_n$.
\item (Global) $\left(\bigwedge_{i=1}^n \neg K(p_i)\right)\vee K(\bot)$.
\item (Global) $\bigwedge_{i=1}^n ((\neg T_i)\rightarrow K(\neg T_i))$, where $T_i$ abbreviates $p_1\vee\cdots\vee p_i$.
\item The rule of global necessitation.
\end{enumerate}
\end{thm}

\begin{proof}
By Lemma~\ref{standardizer}, it suffices to prove consistency of the following system $S'$:
\begin{enumerate}
\item $K(\phi\rightarrow\psi)\rightarrow K(\phi)\rightarrow K(\psi)$.
\item $T_n$, by which is meant $p_1\vee\cdots\vee p_n$.
\item $\left(\bigwedge_{i=1}^n \neg K(p_i)\right)\vee K(\bot)$.
\item $\bigwedge_{i=1}^n ((\neg T_i)\rightarrow K(\neg T_i))$, where $T_i$ abbreviates $p_1\vee\cdots\vee p_i$.
\item $K(\phi)$ whenever $\phi$ is valid.
\item $K(\phi)$ whenever $\phi$ is an instance of (1)-(6).
\item $K(\phi)\rightarrow\phi$.
\end{enumerate}
Let $S'_0$ consist of the first six lines of $S'$.
Let $\mathscr{N}$ be a model where $p_1$ is true, $p_i$ is false for $i>1$, and $K$ is interpreted so that
\[
\mbox{$\mathscr{N}\models K(\phi)$ iff $S'_0\models\phi$.}
\]
We will show $\mathscr{N}\models S'$, proving the theorem.  To do that, we need a preliminary claim.

Preliminary Claim: For any $1\leq i\leq n$, $S'_0\not\models p_i$.

Fix $1\leq i\leq n$.
Since $S'_0$ involves no non-modus ponens rules, it suffices to build a model of $S'_0$ where $p_i$ fails.
Since $n>1$, there is some $1\leq j\leq n$ such that $j\not=i$.  Let $\mathscr{M}$ be the model where $p_j$
is true, $p_k$ is false for all $k\not=j$, and knowledge is interpreted so that
\[
\mbox{$\mathscr{M}\models K(\phi)$ for every $\phi$,}
\]
i.e., $\mathscr{M}$ knows everything.
I claim $\mathscr{M}\models S'_0$.  Since $\mathscr{M}\not\models p_i$, this will prove the preliminary claim.
To see $\mathscr{M}\models S'_0$, let $\tau\in S'_0$, we'll show $\mathscr{M}\models \tau$.

Case 1: $\tau$ is $K(\phi\rightarrow\psi)\rightarrow K(\phi)\rightarrow K(\psi)$.
Well, $\mathscr{M}\models K(\psi)$ since $\mathscr{M}$ knows everything.

Case 2: $\tau$ is $T_n$.  Then $\mathscr{M}\models \tau$ since $\mathscr{M}\models p_j$.

Case 3: $\tau$ is $\left(\bigwedge_{i=1}^n \neg K(p_i)\right)\vee K(\bot)$.  Then $\mathscr{M}\models \tau$
since $\mathscr{M}\models K(\bot)$.

Case 4: $\tau$ is $\bigwedge_{i=1}^n ((\neg T_i)\rightarrow K(\neg T_i))$.
It suffices to show $\mathscr{M}\models K(\neg T_k)$ for an arbitrary $1\leq k\leq n$.
This is true because $\mathscr{M}$ knows everything.

Case 5/6: $\tau$ is $K(\phi)$ where $\phi$ is valid or $\phi$ is an instance of (1)-(6).  Then $\mathscr{M}\models K(\phi)$
since $\mathscr{M}$ knows everything.

This shows $\mathscr{M}\models S'_0$, and so since $\mathscr{M}\not\models p_i$, $S'_0\not\models p_i$.  The Preliminary Claim is proved.

Back to the main theorem, to show $\mathscr{N}\models S'$, let $\sigma\in S'$, we'll show $\mathscr{N}\models\sigma$.

Case 1: $\sigma$ is $K(\phi\rightarrow\psi)\rightarrow K(\phi)\rightarrow K(\psi)$.
Suppose $\mathscr{N}\models K(\phi\rightarrow\psi)$ and $\mathscr{N}\models K(\phi)$.
This means $S'_0\models\phi\rightarrow\psi$ and $S'_0\models\phi$.  Thus $S'_0\models\psi$ and $\mathscr{N}\models K(\psi)$.

Case 2: $\sigma$ is $T_n$.  Then $\mathscr{N}\models\sigma$ since $\mathscr{N}\models p_1$.

Case 3: $\sigma$ is $\left(\bigwedge_{i=1}^n \neg K(p_i)\right)\vee K(\bot)$.
To show $\mathscr{N}\models\sigma$ it suffices to let $1\leq i\leq n$ be arbitrary and show $\mathscr{N}\models \neg K(p_i)$.
By the Preliminary Claim, $S'_0\not\models p_i$, thus $\mathscr{N}\not\models K(p_i)$, as desired.

Case 4: $\sigma$ is $\bigwedge_{i=1}^n ((\neg T_i)\rightarrow K(\neg T_i))$.
Then $\mathscr{N}\models\sigma$ vacuously, because\footnote{This technique fails if we want to make $p_1$ false in $\mathscr{N}$,
i.e.~if we want the surprise examination to fall on a day other than Monday.  The proof can be altered, at the price of slightly more
complexity, to allow a surprise exam on any day except for day $n$.  For a surprise exam on day $1<k<n$,
one would add $\neg T_{k-1}$ as a global axiom to the system and adjust the proof accordingly, and in the proof of the Preliminary
Claim, choose $j>k$.}
$\mathscr{N}\not\models \neg T_i$ for any $i$, because $\mathscr{N}\models p_1$.

Case 5: $\sigma$ is $K(\phi)$ where $\phi$ is valid.  Since $\phi$ is valid, $S'_0\models\phi$, so $\mathscr{N}\models K(\phi)$.

Case 6: $\sigma$ is $K(\phi)$ where $\phi$ is an instance of (1)-(6).  Then $\phi\in S'_0$, so $S'_0\models\phi$, so $\mathscr{N}\models K(\phi)$.

Case 7: $\sigma$ is $K(\phi)\rightarrow\phi$.
Suppose $\mathscr{N}\models K(\phi)$.  This means $S'_0\models \phi$.
By Cases 1-6, $\mathscr{N}\models S'_0$.  Since $\mathscr{N}\models S'_0$ and $S'_0\models \phi$,
$\mathscr{N}\models\phi$ and our work is finished.
\end{proof}

\section{Example: Fitch's Paradox}

Fitch's Paradox, first published in \cite{fitch},
is the fact that under certain assumptions,
if every true fact is knowable then every true fact is known.
See \cite{salerno} for an introduction, though we part ways when
it comes to the exact system to use\footnote{I
should confess here that in \cite{alexandersynthese}
I committed a grave sin: I misread Salerno's system and attributed
a different system to him for which he was not responsible.}.  Instead we 
turn to \cite{costa} who obtained the following system by
very systematic means (the language is propositional, extended by modal 
operators $K$ and $\square$, and (if we understand correctly) $\diamond$ 
abbreviates $\neg\square\neg$):
\begin{enumerate}
\item $\square(\phi\rightarrow\psi)\rightarrow 
(\square(\phi)\rightarrow\square(\psi))$.
\item $(K(\phi\rightarrow\psi)\wedge K(\phi))\rightarrow K(\psi)$.
\item $K(\phi)\rightarrow\phi$.
\item The knowability thesis: $\phi\rightarrow\diamond(K(\phi))$.
\item $\square$-necessitation: $\phi/\square(\phi)$.
\item $K$-necessitation: $\phi/K(\phi)$.
\end{enumerate}
From these, the \emph{omniscience principle} schema, $\phi\rightarrow 
K(\phi)$, can be deduced by what is known as the Church-Fitch argument.
This is considered a paradox since it seems plausible that all truths are 
knowable ($\phi\rightarrow\diamond(K(\phi))$) but it seems absurd that we 
are omniscient.  We would like to weaken necessitation to remove this 
unwanted consequence.  It turns out that if we localize
$K(\phi)\rightarrow\phi$, it resolves the paradox so strongly, we can
even strengthen the knowability thesis.

\begin{thm}
\label{fitch1}
Let $\mathscr{L}$ be a propositional language (with at least one atom 
$q$) extended by modal operators $K$ and $\square$.  Let $S$ be the 
following system in $\mathscr{L}$:
\begin{enumerate}
\item (Global) $\square(\phi\rightarrow\psi)\rightarrow 
(\square(\phi)\rightarrow\square(\psi))$.
\item (Global) $(K(\phi\rightarrow\psi)\wedge K(\phi))\rightarrow 
K(\psi)$.
\item (Local) $K(\phi)\rightarrow\phi$.
\item (Global) The strong knowability thesis: 
$\diamond(K(\phi))$.
\item The rule of global necessitation for $\square$.
\item The rule of global necessitation for $K$.
\end{enumerate}
Then $S\not\models q\rightarrow K(q)$.
\end{thm}

\begin{proof}
By a straightforward variation of Lemma~\ref{standardizer}, it suffices to show $S'\not\models q\rightarrow K(q)$, where $S'$ is:
\begin{enumerate}
\item $\square(\phi\rightarrow\psi)\rightarrow(\square(\phi)\rightarrow\square(\psi))$.
\item $(K(\phi\rightarrow\psi)\wedge K(\phi))\rightarrow K(\psi)$.
\item The strong knowability thesis: $\diamond(K(\phi))$.
\item $K(\phi)$ whenever $\phi$ is valid.
\item $\square(\phi)$ whenever $\phi$ is valid.
\item $K(\phi)$ whenever $\phi$ is an instance of (1)-(7).
\item $\square(\phi)$ whenever $\phi$ is an instance of (1)-(7).
\item $K(\phi)\rightarrow\phi$.
\end{enumerate}
Let $S'_0$ consist of lines (1)-(7) of $S'$ and let $\mathscr{N}$ be a model in which $q$ is true and $K$ and $\square$
are interpreted so
\[
\mbox{$\mathscr{N}\models K(\phi)$ iff $S'_0\models\phi$, and $\mathscr{N}\models\square(\phi)$ iff $S'_0\models\phi$}\]
(we admit this is a very unusual semantics for $\square$, but since the theorem we are proving is entirely syntactical,
that should not matter).  We will show $\mathscr{N}\models S'$ and $\mathscr{N}\not\models q\rightarrow K(q)$.
Since $S'$ involves no non-modus ponens rules of inference, this 
will show $S'\not\models q\rightarrow K(q)$ and prove the theorem.

Preliminary Claim:  There is a model $\mathscr{M}$ such that $\mathscr{M}\models S'_0$, $\mathscr{M}\not\models q$,
and $\mathscr{M}\not\models \neg K(\phi)$ for every formula $\phi$.

Let $\mathscr{M}$ be a model in which $q$ is false, $K$ is interpreted so that
\[
\mbox{$\mathscr{M}\models K(\phi)$ for every $\phi$,}
\]
and $\square$ is interpreted recursively so that
\[
\mbox{$\mathscr{M}\models \square(\phi)$ iff $\mathscr{M}\models\phi$.}
\]
Right away $\mathscr{M}\not\models q$ and $\mathscr{M}\not\models\neg K(\phi)$, it remains to show
$\mathscr{M}\models S'_0$.  Let $\tau\in S'_0$, we will show $\mathscr{M}\models\tau$.

Case 1: $\tau$
is $\square(\phi\rightarrow\psi)\rightarrow (\square(\phi)\rightarrow 
\square(\psi))$.
The way $\mathscr{M}$ interprets $\square$, to show 
$\mathscr{M}\models\tau$
we must show $\mathscr{M}\models 
(\phi\rightarrow\psi)\rightarrow(\phi\rightarrow\psi)$, but this is a 
tautology.

Case 2: $\tau$
is $(K(\phi\rightarrow\psi)\wedge K(\phi))\rightarrow K(\psi)$.
Then $\mathscr{M}\models\tau$ because $\mathscr{M}\models K(\psi)$.

Case 3: $\tau$ is $\diamond(K(\phi))$.
Unwrapping the abbreviation, $\tau$ is
$\neg\square(\neg K(\phi))$.
By the way $\mathscr{M}$ interprets $\square$,
we must show
$\mathscr{M}\models \neg(\neg K(\phi))$,
i.e.~$\mathscr{M}\models K(\phi)$, which is true by the way
$\mathscr{M}$ interprets $K$.

Case 4/5: $\tau$ is $K(\phi)$ or $\square(\phi)$, where $\phi$ is valid.
Trivial.

Case 6: $\tau$ is $K(\phi)$ where $\phi$ is an instance of (1)-(7).
Trivial.

Case 7: $\tau$ is $\square(\phi)$ where $\phi$ is an instance of (1)-(7).
Then $\mathscr{M}\models\tau$ by Lemma~\ref{techlemma}.

This shows $\mathscr{M}\models S'_0$ and proves the Preliminary Claim.

Back to the main theorem, we have two things to show: that 
$\mathscr{N}\not\models q\rightarrow K(q)$, and that $\mathscr{N}\models 
S'$.
Since $\mathscr{N}\models q$, to show $\mathscr{N}\not\models q\rightarrow 
K(q)$ we must show $\mathscr{N}\not\models K(q)$, i.e., that 
$S'_0\not\models q$.
But this is true by the Preliminary Claim: there is a model where $S'_0$ 
holds and $q$ fails.

To show $\mathscr{N}\models S'$, let $\sigma\in S'$, we will show 
$\mathscr{N}\models S'$.

Case 1/2: $\sigma$ is $\square(\phi\rightarrow\psi)\rightarrow 
(\square(\phi)\rightarrow \square(\psi))$
or $(K(\phi\rightarrow\psi)\wedge K(\phi))\rightarrow K(\psi)$.  By now, 
this case should be 
straightforward.

Case 3: $\sigma$ is $\diamond(K(\phi))$.
That is, $\sigma$ is $\neg \square(\neg K(\phi))$.
By the Preliminary Claim, there is a structure
where $S'_0$ holds and $\neg K(\phi)$ fails.
Since there are no non-modus ponens rules of inference in $S'_0$,
this shows $S'_0\not\models\neg K(\phi)$.
Thus $\mathscr{N}\not\models \square(\neg K(\phi))$, as desired.

Case 4/5: $\sigma$ is $K(\phi)$ or $\square(\phi)$, where $\phi$ is valid.
Trivial.

Case 6/7: $\sigma$ is $K(\phi)$ or $\square(\phi)$, where $\phi$ is an 
instance of (1)-(7).  Then $\phi\in S'_0$, thus $S'_0\models \phi$, thus 
$\mathscr{N}\models K(\phi)\wedge \square(\phi)$.

Case 8: $\sigma$ is $K(\phi)\rightarrow\phi$.
Suppose $\mathscr{N}\models K(\phi)$, which means $S'_0\models\phi$.
By Cases 1-7, $\mathscr{N}\models S'_0$.  Thus $\mathscr{N}\models\phi$.
This proves the theorem.
\end{proof}

The above theorem is strictly stronger than the second theorem of \cite{alexandersynthese},
in which the knowability thesis was also localized along with $K(\phi)\rightarrow\phi$.

To prove Theorem~\ref{fitch1},
we constructed a model in which $K$ and $\square$
were interpreted identically.
Thus we could strengthen Theorem~\ref{fitch1} by adding
$K(\phi)\leftrightarrow\square(\phi)$ as a local axiom
(call it $K=\square$).
Could we add this axiom \emph{globally}?
The answer turns out to be ``no'';
one can see this by using the variation on the Church-Fitch
argument which I published in \cite{alexanderreasoner}.
It is, however, possible to globalize
$K=\square$ and avoid paradox if, in exchange, the knowability thesis is localized:

\begin{thm}
\label{fitch2}
The following system does not prove $q\rightarrow K(q)$:
\begin{enumerate}
\item (Global) $\square(\phi\rightarrow\psi)\rightarrow
(\square(\phi)\rightarrow\square(\psi))$.
\item (Global) $(K(\phi\rightarrow\psi)\wedge K(\phi))\rightarrow
K(\psi)$.
\item (Global) $K(\phi)\leftrightarrow \square(\phi)$.
\item (Local) $K(\phi)\rightarrow\phi$.
\item (Local) The strong knowability thesis:
$\diamond(K(\phi))$.
\item The rule of global necessitation for $\square$.
\item The rule of global necessitation for $K$.
\end{enumerate}
\end{thm}

\begin{proof}[Proof Sketch]
Similar to the proof of Theorem~\ref{fitch1}, with $S'$ and $S'_0$ modified in the obvious ways.
The crucial difference is the proof of the Preliminary Claim.
When constructing the model $\mathscr{M}$ of the Preliminary Claim, make it interpret $K$ and $\square$
so that
\[
\mbox{$\mathscr{M}\models K(\phi)$ for all $\phi$, and $\mathscr{M}\models\square(\phi)$ for all $\phi$.}
\]
Note that $\mathscr{M}$ will then fail the knowability thesis-- but that is fine, since the knowability thesis is no longer in $S'_0$--
but satisfy $K(\phi)\leftrightarrow\square(\phi)$.
\end{proof}

In summary we can say the following about the structure of Fitch's paradox ($G$ stands for global, $L$ stands for local):
\begin{itemize}
\item $\mbox{$G$-Soundness} + \mbox{$L$-Knowability} \models \mbox{Paradox}$ (\cite{fitch}).
\item $\mbox{$L$-Soundness} + \mbox{$G$-Knowability} \not\models \mbox{Paradox}$ (Theorem~\ref{fitch1}).
\item $\mbox{$G$-Knowability} + \mbox{$G$-$(K=\square)$} \models \mbox{Paradox}$ (\cite{alexanderreasoner}).
\item $\mbox{$G$-$(K=\square)$} + \mbox{$L$-Knowability} + \mbox{$L$-Soundess} \not\models \mbox{Paradox}$ (Theorem~\ref{fitch2}).
\end{itemize}
In \cite{halbach2008} we were asked, ``What if the preferred remedy for the paradox of the Knower and other paradoxes
of self-reference also helps to resolve Fitch's paradox?''  In light of Theorems~\ref{predicates4} and \ref{fitch1} it seems
we've achieved precisely that, at least if \emph{preferred} is read as ``preferred by the present author''.

\section{Example: Machines which know their own codes}

It is considered a well-known fact that if a mechanical knowing agent is 
capable of logic and arithmetic and self-reflection, then that machine 
cannot know the index of a Turing machine representing it.
See \cite{lucas}, \cite{benacerraf}, \cite{reinhardt},
\cite{penrose}, \cite{carlson}, and \cite{putnam}\footnote{Putnam does not speak so much about arbitrary machines which know their own code,
but rather argues that the totality of scientific knowledge is not a machine which knows its own code.}.
However, the proofs
involve unrestricted necessitation in various guises.
We will weaken necessitation and explicitly construct a machine which 
knows its own code.

Following \cite{carlson}, we work in the language of Peano Arithmetic 
extended by a modal operator $K$ (i.e., the language of Epistemic 
Arithmetic of \cite{shapiro}).  For \emph{de re} semantics we use Carlson's \emph{base 
logic}.  To paraphrase and summarize:
a structure for a modal language consists of a first-order structure 
$\mathscr{N}$ for its first-order part, together with a function which 
takes a purely modal formula $K(\phi)$ and a variable assignment $s$ and 
outputs True or False-- in which case we write $\mathscr{N}\models 
K(\phi)[s]$ or $\mathscr{N}\not\models K(\phi)[s]$ respectively-- 
satisfying some technical conditions which we will here brush under the 
rug.

If $\phi$ is a formula with free variables $x_1,\ldots,x_n$ (and no 
others),
a \emph{universal closure} of $\phi$ is $\forall x_1 \cdots \forall x_n 
\phi$.  To prove that a given model $\mathscr{N}$ satisfies a universal 
closure of 
$\phi$, it suffices to let $s$ be an arbitrary assignment and show $\mathscr{N}\models\phi[s]$.

It 
follows from \cite{carlson} (specifically from Proposition 3.2, 
Definitions 3.1 and 3.4, and the discussion on p.~54) that the following 
system is inconsistent (this can be glossed as ``a machine cannot know its 
own code''):
\begin{enumerate}
\item $S4$, by which is meant
\begin{enumerate}
\item Universal closures of $K(\phi\rightarrow\psi)\rightarrow 
K(\phi)\rightarrow K(\psi)$.
\item Universal closures of $K(\phi)\rightarrow\phi$.
\item Universal closures of $K(\phi)\rightarrow K(K(\phi))$.
\item The rule of necessitation.
\end{enumerate}
\item The axioms of Epistemic Arithmetic (i.e., Peano Arithmetic, with the 
induction schema extended to our modal language).
\item (Knowledge of Code)
$\exists e K(\forall x(K(\phi)\leftrightarrow x\in W_e))$, when $x$ is the 
lone free variable of $\phi$.
\begin{itemize}
\item Here $W_e$ is the $e$th r.e.~set and $x\in W_e$ abbreviates a 
formula stating that $x$ lies therein.
\end{itemize}
\end{enumerate}

We will localize (1b) and show the resulting system is 
$\omega$-consistent by explicitly constructing a machine which knows its 
own code (via Kleene's Recursion Theorem).
Our machine construction method is somewhat
similar
to the method
of \cite{carlson} and \cite{carlson2012}.
Whereas my goal is a truthful machine which knows its own code at the cost of
its own truth, Carlson's goal was a truthful machine which knew its own truth
and knew it \emph{had some} code, articulated by Reinhardt's Strong Mechanistic Thesis:
$K(\exists e\forall x(K(\phi)\leftrightarrow x\in W_e))$.
To do this Carlson used a very careful analysis of ordinal arithmetic (\cite{carlson1999}),
later organized into patterns of resemblance in \cite{carlson2001}.
The reason Carlson's goal took so much work was precisely because his machines
knew their own truthfulness globally; by localizing that, we'll have much less
difficulty reaching our own goal.

For every $n\in\mathbb{N}$, let $S_n$ be the following system:
\begin{enumerate}
\item Weakened $S4$, by which is meant
\begin{enumerate}
\item (Global) Universal closures of $K(\phi\rightarrow\psi)\rightarrow 
K(\phi)\rightarrow K(\psi)$.
\item (Local) Universal closures of $K(\phi)\rightarrow\phi$.
\item (Global) Universal closures of $K(\phi)\rightarrow K(K(\phi))$.
\item The rule of global necessitation.
\end{enumerate}
\item (Global) The axioms of Epistemic Arithmetic (i.e., Peano 
Arithmetic, with induction extended to the modal language).
\item (Global) \emph{(Having Code $n$)}
$\forall x(K(\phi)\leftrightarrow \langle 
x,\ulcorner\phi\urcorner\rangle\in W_{\overline{n}})$, provided $x$ is the 
lone free variable in $\phi$.
\begin{itemize}
\item Here $\langle x,\ulcorner\phi\urcorner\rangle\in W_{\overline{n}}$
abbreviates a formula saying $\langle x,gn(\phi)\rangle\in W_n$ where 
$gn(\phi)$ is the G\"{o}del number of $\phi$ and 
$\langle\bullet,\bullet\rangle:\mathbb{N}^2\to\mathbb{N}$ is a canonical 
computable pairing function.
\end{itemize}
\end{enumerate}

My goal is to show that there is some $n\in\mathbb{N}$
such that $S_n$ is $\omega$-consistent.
Assuming I can do that much, we can combine global necessitation with 
(3) to see
\[
S_n\models K(\forall x(K(\phi)\leftrightarrow\langle 
x,\ulcorner\phi\urcorner\rangle\in W_{\overline{n}})),
\]
which implies
$S_n\models \exists e K(\forall x(K(\phi)\leftrightarrow x\in W_e))$.

With Lemma~\ref{standardizer} in mind, for each $n\in\mathbb{N}$, let 
$S'_n$ be the system
\begin{enumerate}
\item Universal closures of $K(\phi\rightarrow\psi)\rightarrow 
K(\phi)\rightarrow K(\psi)$.
\item Universal closures of $K(\phi)\rightarrow K(K(\phi))$.
\item The axioms of Epistemic Arithmetic.
\item $\forall x(K(\phi)\leftrightarrow \langle 
x,\ulcorner\phi\urcorner\rangle\in W_{\overline{n}})$ ($x$ the lone free 
variable of $\phi$).
\item Universal closures of $K(\phi)$ whenever $\phi$ is valid (true in 
every structure according to the base logic).
\item $K(\phi)$ whenever $\phi$ is an instance of (1)-(6).
\item Universal closures of $K(\phi)\rightarrow\phi$.
\end{enumerate}

\begin{lem}
\label{standardizer2}
For each $n\in\mathbb{N}$, and each sentence $\phi$,
if $S_n\models\phi$, then $S'_n\models\phi$.
\end{lem}

\begin{proof}
Similar to the proof of Lemma~\ref{standardizer}
(Carlson points out that the base logic satisfies the compactness
theorem).
\end{proof}

Thus, we only need show that $S'_n$ is $\omega$-consistent for some 
$n\in\mathbb{N}$.

If $\phi$ is a formula and $s$ is an assigment, let $\phi^s$ denote the 
sentence
\[
\phi(x_1|\overline{s(x_1)})(x_2|\overline{s(x_2)})\cdots
\]
obtained by replacing every free variable in $\phi$ by the numeral
for its $s$-value.

For every $n\in\mathbb{N}$, let $S'_{n0}$ be the following set of axioms:
\begin{enumerate}
\item Lines (1)-(6) of $S'_n$.
\item (Assigned Validity) $\phi^s$, for valid $\phi$ and $s$ an assignment.
\item $K(\phi)$ for each instance of (1)-(3) of $S'_{n0}$.
\end{enumerate}

\begin{lem}
\label{thereisanf}
There is a total computable function $f:\mathbb{N}\to\mathbb{N}$ such that 
for every $n\in\mathbb{N}$,
\[
W_{f(n)}
=
\{\langle m,gn(\phi)\rangle\in\mathbb{N}
\,:\,\mbox{$\phi$ is a formula with $FV(\phi)\subseteq \{x\}$
and $S'_{n0}\models \phi(x|\overline{m})$}\}
\]
where $\phi(x|\overline{m})$ is the result of substituting 
the numeral $\overline{m}$ for $x$ in $\phi$.
\end{lem}

\begin{proof}
By the Church-Turing Thesis (Carlson points out that the base logic 
satisfies the completeness theorem; combined with compactness, this 
implies that the set of consequences of each r.e.~theory is r.e.,
and the code of the set of consequences of a theory depends uniformly
computably on the code of the theory).
\end{proof}

\begin{cor}
\label{thereisann}
There is an $n\in\mathbb{N}$ such that
\[
W_n
=
\{\langle m,gn(\phi)\rangle\in\mathbb{N}
\,:\,\mbox{$\phi$ is a formula with $FV(\phi)\subseteq \{x\}$
and $S'_{n0}\models \phi(x|\overline{m})$}\}.
\]
\end{cor}

\begin{proof}
By Kleene's Recursion Theorem and Lemma~\ref{thereisanf}.
\end{proof}

\begin{thm}
\label{thereisamachine}
Let $n$ be as in Corollary~\ref{thereisann}.
Then $S_n$ is $\omega$-consistent.
In other words, there is a knowing machine which knows
that it has code $n$.
\end{thm}

\begin{proof}
As discussed above, it suffices to show $S'_n$ is $\omega$-consistent.
Let $\mathscr{N}$ be a structure, in the language of Epistemic Arithmetic,
with universe $\mathbb{N}$, which interprets symbols of PA in the intended 
ways, and which interprets knowledge according to
\[
\mbox{$\mathscr{N}\models K(\phi)[s]$ iff $S'_{n0}\models\phi^s$.}
\]
We will show that $\mathscr{N}\models S'_n$.  First we need a couple 
preliminary claims.

Claim 1: For each formula $\phi$ and assignment $s$,
$\mathscr{N}\models \phi[s]$ iff $\mathscr{N}\models\phi^s$.

By induction on formula complexity.
The only interesting case is when $\phi$ is $K(\phi_0)$.
If $\mathscr{N}\models K(\phi_0)[s]$, by definition this
means $S'_{n0}\models \phi^s_0$.  Since $\phi^s_0$ is a sentence,
for every assignment $t$ we have $(\phi^s_0)^t=\phi^s_0$
and thus $S'_{n0}\models (\phi^s_0)^t$.
This shows that for every such $t$,
$\mathscr{N}\models K(\phi^s_0)[t]$.  By arbitrariness of $t$,
this shows $\mathscr{N}\models K(\phi^s_0)$.
Clearly $K(\phi^s_0)\equiv K(\phi_0)^s$.  The reverse direction-- that if 
$\mathscr{N}\models K(\phi_0)^s$ then $\mathscr{N}\models K(\phi_0)[s]$-- 
is similar.

Claim 2: For each sentence $\phi$, if $S'_{n0}\models \phi$, then
$S'_{n0}\models K(\phi)$.

By compactness, if $S'_{n0}\models\phi$, there are 
$\sigma_1,\ldots,\sigma_{\ell}\in S'_{n0}$ such that
\[
\sigma_1\rightarrow\cdots\rightarrow\sigma_{\ell}\rightarrow\phi\]
is valid.  By (5) of $S'_n$,
\[
S'_{n0} \models
K( \sigma_1\rightarrow\cdots\rightarrow\sigma_{\ell}\rightarrow\phi ).
\]
By repeated application of (1) of $S'_n$,
\[
S'_{n0}\models
K(\sigma_1) \rightarrow \cdots \rightarrow K(\sigma_{\ell}) \rightarrow 
K(\phi).
\]
Since each $\sigma_i\in S'_{n0}$, each $K(\sigma_i)\in S'_{n0}$ since 
$S'_{n0}$ is closed under $K$.  By modus ponens, $S'_{n0}\models K(\phi)$, 
as desired.

Armed with these claims, we are ready to show $\mathscr{N}\models S'_n$.
Let $\sigma\in S'_n$, we will show $\mathscr{N}\models\sigma$.

Case 1: $\sigma$ is a universal closure of 
$K(\phi\rightarrow\psi)\rightarrow K(\phi)\rightarrow K(\psi)$.
Let $s$ be an assignment and assume $\mathscr{N}\models 
K(\phi\rightarrow\psi)[s]$ and $\mathscr{N}\models K(\phi)[s]$.
This means $S'_{n0}\models (\phi\rightarrow\psi)^s$ and $S'_{n0}\models 
\phi^s$.
Clearly $(\phi\rightarrow\psi)^s\equiv \phi^s\rightarrow\psi^s$, so by 
modus ponens, $S'_{n0}\models\psi^s$, so $\mathscr{N}\models K(\psi)[s]$ 
as desired.

Case 2: $\sigma$ is a universal closure of $K(\phi)\rightarrow 
K(K(\phi))$.
To get $\mathscr{N}\models\sigma$, let $s$ be an arbitrary assignment and 
assume $\mathscr{N}\models K(\phi)[s]$.
This means $S'_{n0}\models\phi^s$.
By Claim 2, $S'_{n0}\models K(\phi^s)$, and clearly
$K(\phi^s)\equiv K(\phi)^s$, so $S'_{n0}\models K(\phi)^s$.
This shows $\mathscr{N}\models K(K(\phi))[s]$, as desired.

Case 3: $\sigma$ is an axiom of Epistemic Arithmetic.  If $\sigma$ is a 
\emph{basic axiom} (such as $\forall x(S(x)\not=0)$
or $\forall x\forall y (x\cdot S(y)=x\cdot y+x)$), then 
$\mathscr{N}\models \sigma$ simply because $\mathscr{N}$ has universe 
$\mathbb{N}$ and interprets symbols of PA in the intended ways.
But suppose $\sigma$ is a universal closure of an instance
\[
\phi(x|0)
\rightarrow (\forall x(\phi\rightarrow\phi(x|S(x))))
\rightarrow \forall x\phi
\]
of the induction schema ($\phi$ is allowed to involve $K$).
Let $s$ be an assignment and assume
$\mathscr{N}\models \phi(x|0)[s]$ and
$\mathscr{N}\models \forall x(\phi\rightarrow \phi(x|S(x)))[s]$.
We must show $\mathscr{N}\models \forall x \phi[s]$.

Since $\mathscr{N}\models \phi(x|0)[s]$, it follows by Claim 1 that 
$\mathscr{N}\models \phi(x|0)^s$.  Clearly $\phi(x|0)^s\equiv 
\phi^{s(x|0)}$, where $s(x|0)$ is the assignment which agrees with $s$ 
except that it maps $x$ to $0$.
So $\mathscr{N}\models \phi^{s(x|0)}$.

For each $m\in\mathbb{N}$, since
$\mathscr{N}\models \forall x(\phi\rightarrow\phi(x|S(x)))[s]$,
in particular
$\mathscr{N}\models \phi\rightarrow \phi(x|S(x))[s(x|m)]$.
And thus \emph{if} $\mathscr{N}\models \phi[s(x|m)]$, then
$\mathscr{N}\models \phi(x|S(x))[s(x|m)]$.
By Claim 1, that last sentence can be rephrased: \emph{if}
$\mathscr{N}\models \phi^{s(x|m)}$, then
$\mathscr{N}\models \phi(x|S(x))^{s(x|m)}$; but clearly
$\phi(x|S(x))^{s(x|m)}\equiv \phi^{s(x|m+1)}$,
so in summary so far:
\begin{itemize}
\item $\mathscr{N}\models \phi^{s(x|0)}$.
\item For each $m\in\mathbb{N}$,
if $\mathscr{N}\models \phi^{s(x|m)}$, then
$\mathscr{N}\models \phi^{s(x|m+1)}$.
\end{itemize}
Therefore, by mathematical induction, $\mathscr{N}\models\phi^{s(x|m)}$ 
for every $m\in\mathbb{N}$.
This shows (via Claim 1) that for every $m\in\mathbb{N}$,
$\mathscr{N}\models\phi[s(x|m)]$.
This means precisely that $\mathscr{N}\models\forall x\phi[s]$,
as desired.

Case 4:
$\sigma$ is $\forall x(K(\phi)\leftrightarrow\langle 
x,\ulcorner\phi\urcorner\rangle\in W_{\overline{n}})$, where $x$ is the 
lone free variable in $\phi$.
To show $\mathscr{N}\models\sigma$
it suffices to let $m\in\mathbb{N}$ be arbitrary and prove
$\mathscr{N}\models K(\phi)\leftrightarrow \langle 
x,\ulcorner\phi\urcorner\rangle\in W_{\overline{n}}[s]$ where $s$ is an
assignment with $s(x)=m$.
The following are equivalent:
\begin{align*}
\mathscr{N} &\models K(\phi)[s]\\
S'_{n0} &\models \phi^s & \mbox{(Definition of $\mathscr{N}$)}\\
S'_{n0} &\models \phi(x|\overline{m}) & \mbox{(Since $\phi$ has lone free variable $x$)}\\
\langle m,gn(\phi)\rangle &\in W_n &\mbox{(By choice of $n$ 
(Corollary~\ref{thereisann}))}\\
\mathscr{N} &\models \langle \overline{m},\ulcorner\phi\urcorner\rangle\in W_{\overline{n}}
  &\mbox{(Since $\mathscr{N}$ has standard first-order part)}\\
\mathscr{N} &\models (\langle x,\ulcorner\phi\urcorner\rangle\in W_{\overline{n}})^s &\mbox{(Since $s(x)=m$)}\\
\mathscr{N} &\models \langle x,\ulcorner\phi\urcorner\rangle\in W_{\overline{n}}[s], &\mbox{(By Claim 1)}
\end{align*}
as desired.

Case 5:
$\sigma$ is a universal closure of $K(\phi)$ where $\phi$
is valid.
Let $s$ be any assignment.  Since $\phi$ is valid,
$\phi^s$ is an instance of the Assigned Validity schema from the 
definition of $S'_{n0}$.  Thus $S'_{n0}\models \phi^s$,
hence $\mathscr{N}\models K(\phi)[s]$.

Case 6:
$\sigma$ is $K(\phi)$ where $\phi$ is an instance of (1)-(6)
of $S'_n$.
Then $\phi$ is a sentence (note: this is where we finally reap
our reward for putting all those universal closures everywhere).
So for any assignment $s$, $\phi\equiv\phi^s$,
and so, since (1)-(6) of $S'_n$ are included in $S'_{n0}$,
$S'_{n0}\models\phi^s$, whence $\mathscr{N}\models\phi[s]$.

Case 7:
$\sigma$ is a universal closure of $K(\phi)\rightarrow\phi$.
Let $s$ be an assignment and assume $\mathscr{N}\models K(\phi)[s]$ (so $S'_{n0}\models \phi^s$); we 
must show $\mathscr{N}\models\phi[s]$.

First I claim that $\mathscr{N}\models S'_{n0}$.  To see this, let 
$\tau\in S'_{n0}$, we'll show $\mathscr{N}\models\tau$.

Subcase 1: $\tau$ is an instance of (1)-(6) of $S'_n$.
Then $\mathscr{N}\models\tau$ by Cases 1--6.

Subcase 2: $\tau$ is $\psi^t$ where $\psi$ is valid and $t$ is an 
assignment.  Since $\psi$ is valid, $\mathscr{N}\models \psi[t]$.  By 
Claim 1, $\mathscr{n}\models \psi^t$.

Subcase 3: $\tau$ is $K(\psi)$ where $\psi$ is an instance of (1)-(3) of 
$S'_{n0}$.  Similar to Case 6 above.

This shows $\mathscr{N}\models S'_{n0}$.  Since $\mathscr{N}\models 
S'_{n0}$ and $S'_{n0}\models \phi^s$, $\mathscr{N}\models\phi^s$.  By 
Claim 1, $\mathscr{N}\models\phi[s]$, as desired.  The theorem is proved.
\end{proof}

What this section has established is that there is a certain dichotomy in 
machine knowledge:
\begin{quote}
A truthful machine can know its own code, or it can know its own 
truthfulness, but not both.
\end{quote}

In future work we will explore formulations of Theorem~\ref{thereisamachine}
in which $K$ is a predicate symbol rather than a modal operator.
By localizing soundness, we rule out the self-referential paradoxes which
would normally make predicate symbols unimaginable in this role.
See \cite{halbach} for a discussion of the relative merits of operators vs.~predicates.

\section{On a question of \'{E}gr\'{e} and van Benthem}

At the end of \cite{egre}, the following question appears (attributed to 
J.~van Benthem):
\begin{quote}
``...from what we saw, schematic L\"{o}b's theorem appears as the positive 
counterpart of a negative result ... Could there be other L\"{o}b-style 
strengthenings of the inconsistency results we have presented, that 
provide us with positive information about consistent subsystems of the 
conflicting schemata?''
\end{quote}

Imagine we did not know L\"{o}b's theorem.  One way we might find it is as 
follows.
In Theorem~\ref{predicates4} we
proved consistency of a weak predicate form of $S4$, by constructing a 
specific model.
To probe for things that can be consistently added to the system,
rather than consider the class of all models of the system (which is very 
hard to get our hands on), we now have a particular model.  Any schema 
which holds in the particular model must necessarily be consistent with 
the system, \emph{locally}; whether or not it is consistent
\emph{globally}
is another matter, which requires special attention.
One property we might notice
about the particular model is that it 
satisfies L\"{o}b's theorem.
By asking, ``can this 
newfound consistent local axiom be made global,'' we might discover 
L\"{o}b's schema.

More generally, any time we resolve a paradox by the methods of this 
paper, building a particular model, we can look for new consistent 
local axioms by simply examining that particular model; we can then 
attempt to prove (or disprove) they are consistent globally.  For 
example, as seen at the end of Section 4, this kind of reasoning lead us 
to a new variation of the Church-Fitch argument.

If we find a new globally consistent axiom in this way, it might be a 
candidate to answer \'{E}gr\'{e}'s and van Benthem's question.
If that axiom somehow explains why a stronger system is inconsistent, 
even better.

Here are two concrete candidates to answer the question\footnote{For another possible answer,
see Theorem 2 of \cite{friedman75}.}.
First, the schema $\left(\bigwedge_{i=1}^n \neg K(p_i)\right)\vee K(\bot)$ 
from Section 3.
This author initially resolved the surprise examination paradox
by localizing both $K(\phi)\rightarrow\phi$ and
$\bigwedge_{i=1}^n \neg K(p_i)$, but this was unsatisfactory;
but the schema $\left(\bigwedge_{i=1}^n \neg(p_i)\right)\vee K(\bot)$,
true in the particular model, could be globalized.  The resulting system was a 
``consistent subsystem of the
conflicting schemata'',
and the new schema ``provides us with positive information'' about it.

Likewise the author initially found the schema $\diamond(K(\phi))$ of Section 4
by looking at a specific model witnessing a consistent subsystem of Fitch's
hypotheses, and $\diamond(K(\phi))$ could be argued to provide positive information
about that as well.

Along these same lines, by considering the particular model constructed in 
Section 5,
we might notice it satisfies ($x$ the lone free variable in $\phi$)
\[
(\forall x (\phi\rightarrow K(\phi)))
\rightarrow \exists e \forall x (\phi\leftrightarrow x\in W_e),
\]
and it's not hard to show this schema can consistently be added 
to Theorem~\ref{thereisamachine} as a global axiom.
This schema (and knowledge thereof) is called an \emph{epistemic Church's 
Thesis} in \cite{reinhardtrevista}.

\section{Conclusion}

We have shown that
certain paradoxes vanish,
and certain impossibilities become possible, by
localizing certain axioms.
This should not come as too much of a surprise.  One might 
justify necessitation by saying that if we can prove $\phi$, then the 
knower himself can follow that proof, and therefore knows $\phi$.  But for 
the knower to follow, the proof must only use axioms the knower is aware 
of.  This became very clear in Section 2, where paradox occurred despite 
Myhill's weakening of necessitation: we know by second-order logic that 
Goodstein's Theorem is true in $\mathbb{N}$, but when we don our 
first-order Peano
arithmetist hats, that knowledge lacks conviction.
Merely adding the troublesome axiom would be harmless (as it is true), but 
its presence invalidates the justification for the rule of necessitation.

In most of our examples, we localized soundness, 
$K(\phi)\rightarrow\phi$.  The fact that this resolved all those paradoxes 
gives evidence, in our opinion, that we should take greater care before 
assuming $K(K(\phi)\rightarrow\phi)$.
For further evidence, consider the paradox of the knower of
\cite{kaplan} (as described in \cite{anderson} or on p.~21 of \cite{egre}) 
where $K(\ulcorner K(\ulcorner\phi\urcorner)\rightarrow\phi\urcorner)$ is 
one of just three schemas (plus arithmetic) which lead to paradox, the 
other two schemas being practically unassailable.
And if that is not already a devastating blow against requiring 
$K(K(\phi)\rightarrow\phi)$, \cite{thomason} proves that the following system 
proves $K(\ulcorner\phi\urcorner)$ for every $\phi$ (see p.~23 of 
\cite{egre}):
\begin{enumerate}
\item $K(\ulcorner\phi\urcorner)\rightarrow K(\ulcorner 
K(\ulcorner\phi\urcorner)\urcorner)$,
\item $K(\ulcorner K(\ulcorner\phi\urcorner)\rightarrow\phi\urcorner)$,
\item $K(\ulcorner\phi\urcorner)$ if $\phi$ is valid,
\item $K(\ulcorner\phi\rightarrow\psi\urcorner)\rightarrow 
(K(\ulcorner\phi\urcorner)\rightarrow K(\ulcorner\psi\urcorner))$.
\end{enumerate}
This shows that (at least for predicate symbols)
mere \emph{belief} in soundness is already deadly-- even without soundness 
itself!

See \cite{alexandersynthese} for several arguments\footnote{And see \cite{fraassen} for another
argument of the same, when $K$ represents belief rather than knowledge.}
for the philosophical
plausibility of knowledge which fails $K(K(\phi)\rightarrow\phi)$.
I became even further convinced of this when I read, in \cite{shapiro98}, 
about remarks G\"{o}del made at the Gibbs lecture (\cite{godel}).
The main objection (I think) to admitting the possibility of
$\neg K(K(\phi)\rightarrow\phi)$ is that truthfulness is built into the 
very definition of knowledge.
But in his lecture, G\"{o}del distinguished between objective knowledge 
and subjective knowledge, so that, for example, mathematicians might 
subjectively come to know that mathematics is sound, through some argument 
strictly outside mathematics itself-- an empirical argument, say.  Thus the mathematician might 
understand, subjectively, that $K(\phi)\rightarrow\phi$ holds for her 
knowledge, while rejecting $K(K(\phi)\rightarrow\phi)$ (say, upon 
considering the second incompleteness theorem), not because of doubting 
the \emph{truth} of $K(\phi)\rightarrow\phi$, but rather because of 
doubting the \emph{mathematical provability} of it.

Section 5 is particularly relevant: epistemological traditions should 
assume a back seat when in conflict with practical applications.
A truthful knowing machine which knows its own code could be useful in the 
real world, even without knowing its own truthfulness.  We could, perhaps, 
sidestep the issue by calling such a machine a ``believing machine'' 
rather than a ``knowing machine'', but this would be disingenuous, because 
we know the machine is truthful; to call it a believing machine would be 
to suggest it might be capable of asserting falsities.

\textbf{Further Work.}
One obvious direction to take from here is to expand the list of paradoxes treated
by global necessitation.  One particularly interesting paradox is found on pp.~260--261
of \cite{horsten}.
This paradox can be formally treated with the methods of the current paper,
but as there are no \emph{agents} involved in it, philosophical justification
will require careful treatment in a future paper.

Having obtained, through global necessitation, 
consistent variations of certain paradoxical systems, one thing to 
consider is: which of these consistent systems can be fused while 
preserving consistency? (Compare the work of \cite{friedman} on axioms 
about self-referential truth.)
For example, by merging ideas from the proofs of 
Theorems~\ref{surpriseexamthm} and \ref{fitch1},
it is possible to show that if $n>1$ then the following system is consistent
and does not prove $p_1\rightarrow K(p_1)$, in a 
language 
containing operators $K$, $\square$, and atoms $p_1,\ldots,p_n$, thus 
``simultaneously resolving'' Fitch's and the Surprise Examination 
Paradox: \begin{enumerate}
\item (Global) $K(\phi\rightarrow\psi)\rightarrow K(\phi)\rightarrow 
K(\psi)$.
\item (Global) $T_n$, by which is meant $p_1\vee\cdots\vee p_n$.
\item (Global) $\bigwedge_{i=1}^n((\neg T_i)\rightarrow K(\neg T_i))$, 
where $T_i$ abbreviates $p_1\vee\cdots\vee p_i$.
\item (Global) $\left(\bigwedge_{i=1}^n \neg K(p_i)\right)\vee K(\bot)$.
\item (Global) $\square(\phi\rightarrow\psi)\rightarrow 
(\square(\phi)\rightarrow \square(\psi))$.
\item (Local) The strong knowability thesis: $\diamond(K(\phi))$.
\item (Local) $K(\phi)\rightarrow\phi$.
\item The rules of global necessitation for $K$ and $\square$.
\end{enumerate}
Note that in the above system we've localized 
$\diamond(K(\phi))$; if we kept it global, the same 
consistency proof would not work (the Preliminary Claims
would be hard to 
reconcile).  The consistency or inconsistency of the resulting system is 
presently unknown to me (even if $\diamond(K(\phi))$ were weakened to 
$\phi\rightarrow\diamond(K(\phi))$).

A similar question can be glossed provocatively:
\begin{quote}
Is it possible for a truthful knowing machine to know its own code,
and simultaneously attend a class in which a surprise examination next 
week is announced?
\end{quote}

\vspace*{10pt}

\address{Department of Mathematics\\
\hspace*{9pt}The Ohio State University\\
\hspace*{18pt}231 W.~18th Ave., Columbus, Ohio, 43210, USA\\
{\it E-mail}: alexander@math.ohio-state.edu}

\clearpage

\end{document}